\documentclass{amsart}
\usepackage{bbm}
\usepackage{amsmath,amsthm,amsfonts,amssymb,amscd,mathrsfs}
\usepackage{psfrag,graphicx}

\usepackage{epic,eepic}
\usepackage{color}
\usepackage{ebezier}

\theoremstyle{plain}

\newtheorem{Thm}{Theorem}[section]
\newtheorem{Lem}[Thm]{Lemma}
\newtheorem{Prop}[Thm]{Proposition}

\newtheorem{Cla}[Thm]{Claim}
\theoremstyle{remark}
\newtheorem{Def}[Thm] {Definition}
\newtheorem{Rem}[Thm] {Remark}

\long\def\begcom#1\endcom{}

\newcommand{\length}{\operatorname{\length}}

\def\b{\big}
\def\n{\noindent}
\def\p{\prime}

\def\length{\operatorname{length}}

\def\domination{\operatorname{domination}}

\def\cU{{\mathcal U}}

\def\cM{{\mathcal M}}

\begin{document}

\title[Continuity of  Entropy Map for  Nonuniformly Hyperbolic Systems  ]
      {Continuity of  Entropy Map for  Nonuniformly Hyperbolic Systems}

\author[Liao, Sun, Wang]
{Gang Liao $^1$,  Wenxiang Sun $^{1,2}$, Shirou Wang $^{1}$}

\email{liaogang@math.pku.edu.cn}

\email{sunwx@math.pku.edu.cn}

\email{wangshirou@pku.edu.cn}



\thanks{$^{1}$ School of Mathematical
Sciences,  Peking University, Beijing 100871, China.}

\thanks{$^2$
Sun is supported by National Natural Science Foundation of China (
\# 11231001)}

\date{June, 2014}

 \maketitle



\maketitle

\begin{abstract}
 We prove that entropy map is upper semi-continuous
for $C^{1}$ nonuniformly hyperbolic systems with domination, while
it is not true for $C^{1+\alpha}$ nonuniformly hyperbolic systems in
general. This goes a little against a common intuition that
conclusions are parallel between  $C^{1+\domination}$ systems and
$C^{1+\alpha}$ systems.
\end{abstract}


\section{Introduction}
The  entropy map of a continuous transformation $f$ on a metric
space $M$ is defined by $\mu\rightarrow h_{\mu}(f)$ on the set
$\cM_{inv}(M)$ of all $f$-invariant measures and it is generally not
continuous (see \cite{Mi1}). However, it is still worth our effort to
investigate the upper semi-continuity of entropy map since, for
instance, it implies the existence of invariant measures of maximal
entropy. It has been shown that entropy map is upper semi-continuous
for expansive homeomorphisms of compact metric spaces
\cite{Walters}, and then it is  generalized to entropy expansive maps
\cite{Bowen} as well as asymptotic entropy expansive maps \cite{Mi2}.
In 1989 Newhouse \cite{New1} proved:  (i) for any $C^{\infty}$ maps
the entropy map is upper semi-continuous; (ii) for $C^{1+\alpha}$
 nonuniformly hyperbolic diffeomorphisms the entropy
map, when  restricted on the set of  hyperbolic measures with the
same ``hyperbolic rate", is also upper semi-continuous. In the
present paper,  we remove the assumption on ``hyperbolic rates"  in
\cite{New1} to show that for $C^{1}$ nonuniformly hyperbolic systems
with domination property, the entropy map is upper semi-continuous.

 \begin{Def}\label{def1}
Let $M$ be a compact Riemannian manifold and $f: M\to M$ be a $C^1$
diffeomorphism.  Given $\lambda_{s}, \lambda_{u}\gg\varepsilon>0$,
and for all $k\in \mathbb{N}$, we define
$\Lambda_{k}=\Lambda_{k}(\lambda_{s},\lambda_{u};\varepsilon), k\geq
1$, to be all points $x\in M$ for which there is a splitting
$T_{x}M= E^{s}_{x}\oplus E^{u}_{x}$ with the invariance property
$(D_{x}f^{m})E^{s}_{x}=E^{s}_{f^{m}x}$ and
 $(D_{x}f^{m})E^{u}_{x}=E^{u}_{f^{m}x}$ and satisfying:

 (a) $\|Df^{n}|_{E^{s}_{f^{m}x}}\|\leq e^{\varepsilon k}e^{-(\lambda_{s}-\varepsilon)n}
 e^{\varepsilon |m|}$, $\forall\, m\in{ \mathbb{Z}},\, n\geq1;$

 (b) $\|Df^{-n}|_{E^{u}_{f^{m}x}}\|\leq e^{\varepsilon k}e^{-(\lambda_{u}-\varepsilon)n}e^{\varepsilon |m|}$,
 $\forall\, m\in{ \mathbb{Z}}, \,n\geq1;$

 (c) $tan(ang(E^{s}_{f^{m}x},E^{u}_{f^{m}x}))\geq e^{-\varepsilon k}e^{-\varepsilon|m|}.$

$\Lambda=\Lambda(\lambda_{s},\lambda_{u};\varepsilon)=\bigcup\limits_{k=1}^{+\infty}\Lambda_{k}$
is called a Pesin set.

\end{Def}

Denote by $\cM_{inv}(\Lambda)$ the set of all invariant measures
supported on $\Lambda$, i.e.,
$\mu\in\cM_{inv}(\Lambda)\Leftrightarrow \mu(\Lambda)=1 $. For an
ergodic $f-$invariant measure $\nu$ with non-zero Lyapunov
exponents, we could define a Pesin set associated to it in the
following way. Let $\Omega$ be the Oseledec basin of $\nu$ where all
Lyapunov exponents exist, by Oseledec Theorem \cite{Oseledec}
$\nu(\Omega)=1$. Denote by  $E^{s}$ and $E^{u}$ the direct sum of
the Oseledec splittings with respect to negative and positive
Lyapunov exponents respectively. Let $\lambda_s$ be the norm of the largest
Lyapunov exponent of vectors in $E^{s}$ and $\lambda_u$ be the
smallest one in $E^{u}$ and choose $0<\varepsilon \ll
\min\{\lambda_s, \lambda_u\}.$ Then $\Omega$ is contained in the
Pesin set $\Lambda=\Lambda(\lambda_s, \lambda_u; \varepsilon)$   in
Definition \ref{def1}, i.e.,  $\nu\in \cM_{inv}(\Lambda).$ Observe
that the splitting $T_{x}M=E^{s}(x)\oplus E^{u}(x)$ in Definition
\ref{def1} is not necessary to be continuous with  $x\in \Lambda$,
and the angle between $E^s$ and $E^u$ may approach  to zero along
orbits in $\Lambda$. This discontinuity leads to an obstacle for the
continuity property of entropy map
 on $\mathcal M_{inv}(\Lambda)$. In the present paper,  we make an  assumption
that there is a domination between $E^s$ and $E^u$, which ensures
both continuity of the splittings and the  uniformly bounded angles
below between them. To be precise, a splitting
$T_{x}M=E^{s}(x)\oplus E^{u}(x),\,\,x\in \Lambda$ is dominated if
$\dfrac{\|D_{x}f v\|}{\|D_{x}f u\|}\leq \dfrac{1}{2}$ for any $v\in
E^{s}(x)$ and $u\in E^{u}(x)$ with $\|v\|=1 ,\|u\|=1.$

Here is our main theorem in the paper.

\begin{Thm}\label{thm1}Let $f$ be  a $C^1$ diffeomorphism of a compact Riemannian manifold $M$. Let
$\Lambda=\Lambda(\lambda_{s},\lambda_{u};\varepsilon)$ be a Pesin
set with a dominated splitting  $T_{x}M=E^{s}(x)\oplus E^{u}(x),$
$x\in\Lambda$. Then the entropy map $\mu\mapsto h_{\mu}(f)$ is upper
semi-continuous on $\cM_{inv}(\Lambda)$.
\end{Thm}

\vbox{}Lack of domination may cause no upper semi-continuity of
entropy map for $C^r$ diffeomorphism for any $2\leq r<\infty$ by a
version of  Downarowicz-Newhouse example \cite{New2}. This a little
goes against a common intuition that  conclusions are usually
parallel between $C^{1+\domination}$ systems  and $C^{1+\alpha}$
systems (see for instance, \cite{ABC}\cite{ST}). Moreover, recall that the upper semi-continuity of entropy map is  obtained for $C^1$ diffeomorphisms away from tangencies in \cite{Liao}. However, due to the nonuniformity of hyperbolicity of $(f, \Lambda)$ in Theorem \ref{thm1},   the system $(f,\overline{\Lambda})$  may be approximated by ones having homoclinic  tangencies of some periodic points whose index different from $\dim E^s$, see  for example in section 6.4 of \cite{BV}, where the closure of the Pesin set $\overline{\Lambda}=M$ and $\Lambda$ supports a hyperbolic SRB measure.

In section 2, we recall some definitions and basic facts about
entropy, and give two lemmas  needed to prove Theorem \ref{thm1}. In
section 3, we will prove Theorem \ref{thm1}. By using a class of
$C^{r}$ $(2\leq r <\infty)$ diffeomorphisms studied in \cite{New2}
we illustrate in section 4 that  entropy map could be not upper
semi-continuous for  nonuniformly hyperbolic systems  without
domination.

\section{Preliminaries }\setlength{\parindent}{2em}

Let $M$ be a compact metric space and $f$  a continuous map on $M$.
Let $\mu$ be an $f$-invariant probability measure and
$\xi=\{B_{1},\cdots B_{k}\}$ a finite partition of $M$ into
measurable sets. The entropy of $\xi$ with respect to $\mu$ is
$$H_{\mu}(f,\xi)=-\sum\limits_{i=1}^{k}\mu(B_{i})\log\mu(B_{i}).$$
The entropy of $f$ with respect to $\mu$ and $\xi$ is given by
$$h_{\mu}(f,\xi)=\lim\limits_{n\rightarrow\infty}\frac{1}{n}H_{\mu}(f,\xi^{n})=\inf _{n}\frac{1}{n}H_{\mu}(f,\xi^{n})$$
\noindent where $\xi^{n}=\bigvee\limits_{i=0}^{n-1} f^{-i}\xi$. The
entropy of $f$ with respect to $\mu$ is given by
$$ h_{\mu}(f)=\sup\limits_{\xi} h_{\mu}(f,\xi)$$ \noindent where
$\xi$ is taken over all finite partitions of $M$. A
partition $\alpha=\{A_{0},A_{1},\cdots,A_{k}\}$ is called a compact
partition if $A_{1},\cdots,A_{k}$ are disjoint compact subsets and
$A_{0}=M\backslash \bigcup _{1\leq i\leq k}A_i.$ It is clear  that $
h_{\mu}(f)=\sup\limits_{\alpha} h_{\mu}(f,\alpha)$ , where $\alpha$
is taken over all finite compact partitions of $M$.

Let $F$ be a subset of $M.$  A set $E \subseteq M $ is called a
$(n,\delta)$-spanning set of $F\subseteq M$ with respect to $f$ if $\forall \ x\in F, \
\exists \ y\in E $ such that $d(f^{i}(x),f^{i}(y))\leq \delta, \
0\leq i<n$. Denote $r_{n}(F,\delta,f)$ the minimal cardinality of
sets which $(n,\delta)$-spans $F$ with respect to $f$. Denote
$r(F,\delta,f)=\limsup\limits_{n\rightarrow +\infty}\dfrac{1}{n}\log
r_{n}(F,\delta,f)$  and the topological entropy of $f$ on $F$ is
defined by$$ h_{top}(f,F)=\lim\limits_{\delta\rightarrow
0}r(F,\delta,f) .$$ \noindent In particular, the topological entropy
of $f$ on $M$,  $h_{top}(f,M)$, is denoted  by $h_{top}(f)$.

For each $x\in M$, $n\in \mathbb{N}$,   $r\in \mathbb{R}^+$,  denote
$B_{n}(x,r,f)=\{y\in M: d(f^{i}(x),f^{i}(y))\leq r ,0\leq i <n\}$, and
$ B_{+\infty}(x,r,f)=\{y\in M: d(f^{i}(x),f^{i}(y))\leq r ,i\geq0 \}$.
When $f$ is a homeomorphism, one may also define $B_{\pm
n}(x,r,f)=\{y\in M: d(f^{i}(x),f^{i}(y))\leq r ,-n<i<n\}$ and $
B_{\pm\infty}(x,r,f)=\{y\in M: d(f^{i}(x),f^{i}(y))\leq r ,i\in
\mathbb{Z}\}$. Denote
$$h^{*}_{loc}(x,r,f)=h(f,B_{\pm\infty}(x,r,f)).$$ Further let
$$h_{loc}(x,r,f)=\lim\limits_{\delta\rightarrow
0}\limsup\limits_{n\rightarrow +\infty}\dfrac{1}{n}\log
r_{n}(B_{n}(x,r,f),\delta,f).$$ It's obvious that
$h^{*}_{loc}(x,r,f)\leq h_{loc}(x,r,f)$.

\begin{Lem}\label{lem3}
Let $M$ be a compact Riemannian manifold and $f: M\to M$ be a
diffeomorphism preserving a measure $\mu\in \cM_{inv}(M) $. Then
$$h_{\mu}(f)-h_{\mu}(f, \xi)\leq \int h_{loc}(x,r,f) d\mu(x)$$
\n for any finite partition $\xi$ with $diam(\xi) \leq r.$

\end{Lem}

\begin{proof}\,\,  For a given compact partition
$\alpha=\{A_{0},A_{1},\cdots,A_{k}\},$ where  $A_{1},\cdots,A_{k}$
are disjoint compact subsets and $A_{0}=M\backslash \bigcup _{1\leq
i\leq k}A_i,$ let $$\delta_{0} = \dfrac{1}{2}\min\Big \{d(A_{i},
A_{j}),1\leq i, j\leq k, i\neq j \Big \}.$$  For  $m\in\mathbb{N} $
take $\delta_{1}\in(0,\delta_{0})$ such that $ d(x,y)<\delta_{1}$
implies that $d(f^{i}(x),f^{i}(y))<\delta_{0},\,\, 0\leq i<m.$
Denote $\alpha ^{n}_{f^{m}}=\bigvee\limits_{i=0}^{n-1}
f^{-mi}\alpha.$ Then
\begin{eqnarray*}
& &\dfrac{1}{n} H_{\mu}(\alpha ^{n}_{f^{m}})-\dfrac{1}{n}
H_{\mu}((\xi^{m})^{n}_{f^{m}})\,\,\,\,\,\,\,\,\,(\text{where }\,\,
\xi^{m}=\bigvee\limits_{i=0}^{m-1} f^{-i}\xi)  \\[2mm]
&\leq&\dfrac{1}{n}H_{\mu}(\alpha^{n}_{f^{m}}|(\xi^{m})^{n}_{f^{m}}) \\[2mm]
&=&-\frac{1}{n}\sum\limits_{B\in (\xi^{m})^{n}_{f^{m}}}\mu(B)\sum
\limits_{A\in\alpha^{n}_{f^{m}}}\mu_{B}(A)\log \mu_{B}(A),\ \ \ \
(\text{where}\ \ \mu_{B}(A)=\dfrac{\mu(A)}{\mu({B})})  \\
&\leq &\frac{1}{n}\sum\limits_{B\in (\xi^{m})^{n}_{f^{m}}}\mu(B)\log
N_{B}(\alpha ^{n}_{f^{m}})\ \ \ \ \ \ \ (2.1)
\end {eqnarray*}

\n where $N_{B}(\alpha ^{n}_{f^{m}})=\sharp\{A\in\alpha
^{n}_{f^{m}}:A\cap B\neq\varnothing\}.$

Let $E$ be a $(n,\delta_0)$-spanning set of $B$ with respect to
$f^m$ with minimal cardinality. For every $y\in E,$ by the choice of
$\delta_0,$ the number of elements of $\alpha ^{n}_{f^{m}}$ which
intersect with $B_n(y,\delta_0,f^m)$ could not exceed $2^n.$ Since
$diam(\xi)<r,$ $B\subseteq B_{mn}(x,r)$ for any given $x\in B.$ From
these  we get
\begin{eqnarray*}
N_{B}(\alpha ^{n}_{f^{m}})&\leq& r_{n}(B,\delta_{0},f^{m})\cdot
2^{n} \\
 &\leq& r_{n}(B_{mn}(x,r),\delta_{0},f^{m})\cdot 2^{n}\\
 &\leq& r_{mn}(B_{mn}(x,r),\delta_{1},f)\cdot 2^n
\end{eqnarray*}
\n for any point $x\in B.$ Thus by (2.1) we get
\begin{eqnarray*}
& &\dfrac{1}{n} H_{\mu}(\alpha ^{n}_{f^{m}})-\dfrac{1}{n}
H_{\mu}((\xi^{m})^{n}_{f^{m}})\\
&\leq& \dfrac{m}{mn}\sum\limits_{B\in (\xi^{m})^{n}_{f^{m}}}\int_{B}\log r_{mn}(B_{mn}(x,r),\delta_{1},f)d\mu(x)+\log2 \\
&=& m\int\dfrac{1}{mn}\log
r_{mn}(B_{mn}(x,r),\delta_{1},f)d\mu(x)+\log2.
\end{eqnarray*}
When $n$ is large enough, $\dfrac{1}{mn}\log
r_{mn}(B_{mn}(x,r),\delta_{1},f)$ is less than or equal to
$h_{top}(f)$, which is a finite number for a diffeomorphism on a
compact manifold. Applying  Fatou Lemma we have
\begin{eqnarray*}
& &h_{\mu}(f^{m},\alpha)-h_{\mu}(f^{m},\xi^{m})\\[2mm]
&=&\lim\limits_{n\rightarrow+\infty}\Big(\frac{1}{n} H_{\mu}(\alpha ^{n}_{f^{m}})-\frac{1}{n} H_{\mu}((\xi^{m})^{n}_{f^{m}})\Big)\\[2mm]
&\leq& m\cdot\limsup\limits_{n\rightarrow+\infty}\int\frac{1}{mn}\log r_{mn}(B_{mn}(x,r),\delta_{1},f)d\mu(x)+\log2\\[2mm]
&\leq&m\int\limsup\limits_{n\rightarrow+\infty}\frac{1}{mn}\log r_{mn}(B_{mn}(x,r),\delta_{1},f)d\mu(x)+\log 2\\[2mm]
&\leq& m\int h_{loc}(x,r,f)d\mu(x) +\log2
\end{eqnarray*}
\n for any compact partition $\alpha$ and any $ m\in \mathbb{N}$.
Note that $h_{\mu}(f^{m},\xi^{m})=mh_{\mu}(f,\xi)$, $\forall m\in
\mathbb{N}$. It follows that
$$h_{\mu}(f)-h_{\mu}(f,\xi)\leq \int h_{loc}(x,r,f) d\mu(x).$$

\end{proof}

\begin{Rem}\, The  concept of local entropy
originates from Bowen \cite{Bowen}, where it is used to bound the
difference between metric entropy  and the metric entropy  with
respect to a partition with small diameter. In \cite{Bowen}, the
right-hand side of the inequality is
$$\lim\limits_{\delta\rightarrow 0}\limsup\limits_{n\rightarrow
+\infty}\dfrac{1}{n}\sup\limits_{x\in M}\log
r_{n}(B_{n}(x,r,f),\delta,f),\eqno(2.2)$$
 which is
called the local entropy of $f$. It is obvious that this quantity is
not smaller than $\sup\limits_{x\in M} h_{loc}(x,r,f)$ and thus
$\displaystyle{ \int h_{loc}(x,r,f)d\mu(x)}.$  The quantity
$\displaystyle{\int h_{loc}(x,r,f)d\mu(x)},$ which could be called
local entropy of $(f, \mu),$ is slightly different from
(2.2).  This quantity enables us to deal with local entropy
in an open set(thus a measurable set) which has large measure for
any invariant measure $\nu$ near $\mu$. The hyperbolicity assumption
of measures $\nu$ guarantees ``uniform hyperbolicity"(average along
the orbit) in large measure sets (see Proposition \ref{prop1}) and
thus small local entropy of $(f, \nu).$ In this way we can control
the difference between metric entropy and the metric entropy with
respect to a partition with small diameter for all nearby $\nu$ in
Proposition \ref{claim2}, which is a necessary step to prove Theorem
\ref{thm1}.

\end{Rem}

For a   continuous map $f$ on a compact metric space $M$, a measure
$\nu\in \cM_{inv}(M)$ and a Borel set $A$, by Birkhoff Ergodic
Theorem,  the set of points for which the limit of
$\dfrac{1}{n}\sum\limits^{n-1}_{i=0}\chi_{A}(f^{i}(x))$ exists is
measurable and has measure 1.
\begin{Lem}\label{lem1} For any given $0<\gamma<1,\,\,0<\eta<1 $, there exists $\sigma=\dfrac{1}{2}\gamma\eta$
such that for any measure $\nu\in \cM_{inv}(M)$ and any Borel set
$A$ with $\nu(A)>1-\sigma$ we have
$$\nu\{x:\bar f _{A}(x)>1-\gamma\}>1-\frac{1}{2}\eta,$$
\noindent where $ \bar f
_{A}(x)=\underset{n\rightarrow\infty}\lim\dfrac{1}{n}\sum\limits^{n-1}_{i=0}\chi_{A}(f^{i}(x))$
whenever exists.
\end{Lem}
\proof
By Birkhoff Ergodic Theorem,  $\displaystyle{\int\chi_{A}d\nu(x)}=\displaystyle{\int\bar
f_{A}(x)d\nu(x)}.$ Let $E=\{x:\bar f_{A}(x)>1-\gamma\},$ then
$$\nu(A)=\int_{E}\bar f_{A}(x)d\nu(x)+\int_{M\setminus E}\bar
f_{A}(x)d\nu(x)\leq\nu(E)+(1-\gamma)(1-\nu(E)).$$ Choose
$\sigma=\frac{1}{2}\gamma\eta$, then $\nu(A)>1-\sigma$ implies that
$\nu(E)>1-\frac{1}{2}\eta$. \qed

\begin{Rem} Lemma \ref{lem1} is also true for $f^{-1}.$ \end{Rem}

\section{Proof of Theorem \ref{thm1}}\setlength{\parindent}{2em}

In this section, we prepare several lemmas and propositions and then
prove Theorem \ref{thm1}.

Recall that $f: M\to M$ is a $C^1$ diffeomorphism, which has a Pesin
set $\Lambda=\Lambda(\lambda_s, \lambda_u; \varepsilon)$ with a
dominated splitting $E^s(x)\oplus E^u(x),\,x\in \Lambda$.

\begin{Prop}\label{prop1}  Given  $\mu\in \cM_{inv}(\Lambda)$ and $0<\eta <1$ there exist $\rho>0$ and $L>0$
with the following property. For any $\nu\in B_{\rho}(\mu)\cap
\cM_{inv}(\Lambda)$ there exists a measurable set $T$ with
$\nu(T)>1-\eta$ such that

$$\underset{k\rightarrow+\infty}{\lim}\dfrac{1}{k}\sum\limits^{k}_{i=1}\dfrac{1}{L}\log\b\|Df^{-L}\b|_{E^{u}(f^{iL}(x))}\b\|
<-\lambda_{u}+2\varepsilon, \eqno(1)$$
$$\underset{k\rightarrow+\infty}{\lim}\dfrac{1}{k}\sum\limits^{k}_{i=1}\dfrac{1}{L}\log\b\|Df^{L}\b|_{E^{s}(f^{-iL}(x))}\b\|
<-\lambda_{s}+2\varepsilon\eqno(2)$$

\noindent for any $x\in T$, where $B_{\rho}(\mu)$ denotes the set of
all $f$-invariant measures centered at $\mu$ with radius $\rho.$
\end{Prop}

\begin{proof}\,\,
 For an integer $n\geq 1$ set
$$A^{\varepsilon}_{n}=\Big\{x\in
\Lambda:\,\,\dfrac{1}{m}\log\b\|Df^{-m}\b|_{E^{u}(x)}\b\|<-\lambda_{u}+\varepsilon,
\,\,\forall m\geq n\Big\}.$$  Then $A^{\varepsilon}_{1}\subset\cdots
\subset A^{\varepsilon}_{n}\subset A^{\varepsilon}_{n+1}\subset
\cdots,$ and
$\mu(\bigcup\limits_{n=1}^{\infty}A^{\varepsilon}_{n})=1.$ Let
$$c=\max\Big\{\sup\limits_{x\in
\Lambda}\b\|Df^{-1}\b|_{E^{u}(x)}\b\|,\,\,\sup\limits_{x\in
\Lambda}\b\|Df\b|_{E^{s}(x)}\b\|,\,\,2\Big\}$$ and let $0<\eta<1$ be
given in the condition of the proposition.  Take
$\gamma<\dfrac{\varepsilon}{2\log c}$ with $0<\gamma<1$ and take
$\sigma=\dfrac12 \gamma \eta$ as in Lemma \ref{lem1}. Clearly
$0<\sigma<1.$ Take $ N $ such that
$\mu(\bigcup\limits_{n=1}^{N}A^{\varepsilon}_{n})>1-\sigma$. Let
$$U^{\varepsilon}_{N}=\Big\{x\in
\Lambda:\,\,\dfrac{1}{N}\log\b\|Df^{-N}\b|_{E^{u}(x)}\b\|<-\lambda_{u}+\varepsilon\Big\},$$
then $\bigcup\limits_{n=1}^{N}A^{\varepsilon}_{n}\subseteq
U^{\varepsilon}_{N}.$ Since  $U^{\varepsilon}_{N}$ is open,
$\nu(U^{\varepsilon}_{N})>1-\sigma$ for any $\nu\in
\cM_{inv}(\Lambda)$ close enough to $\mu$. Denote $ \bar f
_{U^{\varepsilon}_{N}}(x)=\underset{n\rightarrow\infty}\lim\dfrac{1}{n}
\sum\limits^{n-1}_{i=0}\chi_{U^{\varepsilon}_{N}}(f^{i}(x))$
whenever exists and $\widetilde{U^{\varepsilon}_{N}}=\Big\{x\in
\Lambda: \bar f_{U^{\varepsilon}_{N}}(x)>1-\gamma\Big\}$. By Lemma
\ref{lem1}, we get that
$\nu(\widetilde{U^{\varepsilon}_{N}})>1-\dfrac{1}{2}\eta$.

Following the same procedure for $f^{-1}$ and
$\dfrac{1}{n}\log\b\|Df^{n}\b|_{E^{s}(x)}\b\|$, we get $N^{\prime}$
and thus $\widetilde{(U^{\varepsilon}_{N^{\prime}})}^{{\prime}}$
with
$\nu(\widetilde{(U^{\varepsilon}_{N^{\prime}})}^{{\prime}})>1-\dfrac{1}{2}\eta$
for $\nu\in \cM_{inv}(\Lambda)$ close to $\mu$. Then we get a set
$T=\widetilde{U^{\varepsilon}_{N}}\bigcap\widetilde{(U^{\varepsilon}_{N^{\prime}})}^{{\prime}}$
and a constant $\rho>0$ such that  $\nu(T)>1-\eta$  for any $\nu\in
B_\rho(\mu)\cap \cM_{inv}(\Lambda).$ Let
$L=2\max\Big\{[\dfrac{2N\log
c}{\varepsilon}],[\dfrac{2N^{\prime}\log c}{\varepsilon}]\Big\}$.

For $x\in T$ and $i\in \mathbb{Z^{+}}$, choose a sequence of
integers $\{n^{i}_{j}\}^{\ell+1}_{j=1},$
$$(i-1)L=n^{i}_{\ell+1}<n^{i}_{\ell}<n^{i}_{\ell-1}<\cdots<n^{i}_{1}=iL$$ by the
following procedure
\begin{equation*}n^{i}_{j+1}=
\begin{cases}n^{i}_{j}-N,&\text{$n^{i}_{j}\geq(i-1)L+N$ and
$f^{n^{i}_{j}}(x)\in
U^{\varepsilon}_{N}$}\\n^{i}_{j}-1,&\text{otherwise}.
\end{cases}
\end{equation*}
\noindent where $1\leq j\leq\ell.$ Write
$\{n^{i}_{1},\cdots,n^{i}_{\ell}\}$ as a disjoint union
$A_{i}\bigcup B_{i}\bigcup C_{i}$, where
\begin{eqnarray*}
&A_{i}&=\b\{n^{i}_{j}\geq(i-1)L+N,\,\
 \text{and }  f^{n^{i}_{j}}(x)\in
U^{\varepsilon}_{N}\b\},\\
&B_{i}&=\b\{n^{i}_{j}\geq(i-1)L+N \,\,\text{ and } \,\,
f^{n^{i}_{j}}(x)\notin U^{\varepsilon}_{N}\b\},\\
&C_{i}&=\b\{(i-1)L<n^{i}_{j}<(i-1)L + N\b\}.
\end{eqnarray*}
 It's obvious that
$0\leq\sharp{C_{i}}< N$ , $0\leq\sharp{A_{i}}\leq[\dfrac{L}{N}]$.
Thus,
\begin{eqnarray*}
& &\log\b\|Df^{-L}\b|_{E^{u}(f^{iL}(x))}\b\| \\[2mm]
&\leq&\sum\limits^{\ell}_{j=1}\log\b\|Df^{-(n^{i}_{j}-n^{i}_{j+1})}\b|_{E^{u}(f^{n^{i}_{j}}(x))}\b\|\\[2mm]
&\leq& N(-\lambda_{u}+\varepsilon)\cdot\sharp A_{i} + \log
c\cdot\sharp B_{i} + \log c\cdot\sharp C_{i} \\[2mm]
&<&(-\lambda_{u}+ \varepsilon)L +\log c \cdot (N+\sharp B_{i}).
\end{eqnarray*}
By the definition of $\widetilde{U^{\varepsilon}_{N}}$, for any $k$
large enough, $\sum\limits^{k}_{i=1}\sharp B_{i}\leq kL\gamma$.
 Therefore,
\begin{eqnarray*}
&
&\dfrac{1}{k}\sum\limits^{k}_{i=1}\dfrac{1}{L}\log\b\|Df^{-L}\b|_{E^{u}(f^{iL}(x))}\b\|\\[2mm]
 &\leq& \dfrac{1}{kL}(kL\cdot (-\lambda_{u}+ \varepsilon) + \log c \cdot Nk\ + \log c \cdot
 kL\gamma)\\[2mm]
 &<& (-\lambda_{u}+ \varepsilon) + \frac{1}{2} \varepsilon + \frac{1}{2} \varepsilon< -\lambda_{u}+ 2\varepsilon.
\end{eqnarray*}
Hence,  $$
\underset{k\rightarrow+\infty}{\lim}\dfrac{1}{k}\sum\limits^{k}_{i=1}\dfrac{1}{L}\log\b\|Df^{-L}\b|_{E^{u}(f^{iL}(x))}\b\|<-\lambda_{u}+2\varepsilon
, \ \ \forall x\in T$$ and we get (1).

Replace $f$ and $E^{u}(x) $ by $f^{-1}$ and $E^{s}(x)$ respectively,
we get (2) analogously.
\end{proof}

The following lemma  comes from Burns and Wilkinson \cite{Burn
Amie}, which uses locally invariant fake foliations to avoid the
assumption of dynamical coherence, a construction that goes back  to
Hirch, Pugh, and Shub \cite{HPS}. Given a foliations $ \mathcal {F}$
and a point $y$ in domain, we denote $ \mathcal {F}$$(y) $ the leaf
through $y$ and by $ \mathcal {F}$$(y,\rho) $ we mean the
neighborhood of radius $\rho>0$ around $y$ inside the leaf.

\begin{Lem}\label{lem2}
Let $f: M\to M$ be a $C^1$ diffeomorphism. Assume that $\Delta$ is
an $f$-invariant compact set and the tangent space of which admits
f-dominated splitting: $T_{x}(M)=E^{s}(x)\oplus E^{u}(x), \forall
x\in \Delta$ . Let angles between $f$-invariant subbundles $E^{s}$
and $E^{u}$ are bounded from zero by $\theta$ for every $x\in
\Delta$. Then for any $0< \zeta< \dfrac{\theta}{4}$, $\exists
\\ \rho>r_{0}>0$, for any $x\in \Delta$, the neighborhood $B(x,\rho)$
admits foliations $ \mathcal {F}^{s}_{x}$ and  $\mathcal
{F}^{u}_{x}$, such that for any $y\in B(x,r_{0})$ and $*\in
\{s,u\}$,

(1) almost tangency: leaf $\mathcal {F}^{*}_{x}(y)$ is $C^{1}$ and
$T_{y}\mathcal {F}^{*}_{x}(y)$ lies in a cone of width $\zeta$ of
$E^{*}(x)$,

(2) local invariance: $f\mathcal {F}^{*}_{x}(y,r_{0})\subseteq
\mathcal {F}^{*}_{fx}(fy)$   , $f^{-1}\mathcal
{F}^{*}_{x}(y,r_{0})\subseteq \mathcal {F}^{*}_{f^{-1}x}(f^{-1}y)$.

\end{Lem}

By Lemma \ref{lem2} (1) one can define local product structure on
the $r$-neighborhoods of every $x\in \Delta$, for a small $r>0,$ as
used in \cite{Liao}.

For $y,z\in B(x,\rho)$, write $[y,z]_{s,u}=a$ if $\mathcal
{F}^{s}_{x}(y)$ intersects $\mathcal {F}^{u}_{x}(z)$ at $a\in
B(x,\rho)$. By transversality of (1), the intersection point $a$ is
unique whenever it exists. We could find $r_{1}\in [0,r_{0}]$
independent of $x$ such that $[y,z]_{s,u}$ is well defined whenever
$y,z\in B(x,r_{1}),$ and for any $y\in B(x,r_{1})$ there exists
$y_{*}\in \mathcal {F}^{*}_{x}(x)$ such that
$[y_{s},y_{u}]_{u,s}=y$. Lemma \ref{lem2} implies that the locally
invariant foliations $\mathcal {F}^{*}_{x}$ are transverse with
angles uniformly bounded from below. Therefore, $\exists \,\,\ell>0$
independent of $x$ such that for any $y\in B(x,r^{\p})$ we have
$y_{*}\in \mathcal {F}^{*}_{x}(x,\ell r^{\p})$ for $\ell
r^{\p}<r_{1}$. Furthermore, by locally invariance of foliations we
get that $y\in B_{\pm2}(x,r^{\p})$ implies that
$f^{\pm1}(y_{*})=(f^{\pm1}y)_{*},$ where recall that
$B_{\pm2}(x,r^{\p})=\{y\in M: d(f^{i}y, f^{i}x)\leq r^{\p},-2<
i<2\}$.

Also note that $y_{s\backslash u}=x$ for $y\in B(x,r^{\p})$ implies
that $y\in\mathcal{F}^{u\backslash s}_{x}(x,\ell r^{\p})$, therefore
$y_{s}=y_{u}=x$ implies that $y=x$ for $y\in B(x,r^{\p}).$

Since there exists domination on the Pesin set
$\Lambda=\Lambda(\lambda_s, \lambda_u; \varepsilon), $ we could
extend it to the closure of $\Lambda,$ then the process above could
be done in $\overline{\Lambda}.$ Therefore, we get $r^{\p}$
independent of $x$ in $\Lambda$ such that  $y\in B(x,r^{\p})$
implies that $y=x$.

\begin{Lem}[Pliss\cite{Pliss}]\label{lempliss}
Let $a_{*}\leq c_{2}<c_{1}$ and
$\theta=\dfrac{c_{1}-c_{2}}{c_{1}-a_{*}}.$  For given real numbers
$a_{1},\cdots, a_{N}$ with $\sum\limits_{i=1}^{N}a_{i}\leq c_{2}N $
and $a_{i}\geq a_{*}$ for every $i$,  there exists $\ell\geq
N\theta$ and $1\leq n_{1}<n_{2}<\cdots < n_{\ell}\leq N $, such that
$$\sum\limits_{i=k+1}^{n_{j}}a_i\leq c_{1}(n_{j}-k) , \quad\quad0\leq k<n_{j}, \ \ 1\leq j\leq \ell.$$
\end{Lem}

\vbox{}By (1) of Proposition \ref{prop1} , for every $x\in T$ and
every $k $ large enough,
$$\sum_{i=1}^{k}\log\b\|Df^{-L}\b|_{E^{u}(f^{iL}(x))}\b\|\leq (-\lambda_{u}+2\varepsilon)Lk.$$
Take $a_{*}=\inf \limits_{x\in \Lambda }\{\log \| Df^{-L}
\mid_{E^{u}(x)}\| \}, \,\, c_{1}=(-\lambda_{u}+2\varepsilon)L, \,\,
c_{2}=(-\lambda_{u}+3\varepsilon )L.$  Applying Lemma \ref{lempliss}
it is easy to find an infinite sequence
$$1\leq n_{1}<n_{2}<\cdots<n_{j}<\cdots$$ such that

$$\sum_{i=k+1}^{n_{j}}\log\b\|Df^{-L}\b|_{E^{u}(f^{iL}(x))}\b\|\leq (-\lambda_{u}+3\varepsilon)(n_{j}-k)L,$$
$0\leq k<n_{j},\,\, j=1,2 \cdots. $

Choose $r^{\p\p}>0$ and $\zeta >0$ such that
$\dfrac{\b\|D_{y}f^{-L}v\b\|}{\b\|D_{x}f^{-L}u\b\|} < e^{\varepsilon
L}$  and $\dfrac{\b\|D_{y}f^{L}v\b\|}{\b\|D_{x}f^{L}u\b\|} < e^{
\varepsilon L}$  for $d(x,y)<r^{\p\p}$, \
$\angle(u,v)\leq\dfrac{\zeta}{2}, \ \|u\|=\|v\|=1$.

Take $r=\min \Big\{{r^{\p}},{r^{\p\p}}\Big\}$, then
$$f^{-(n_{j}-k)L}\mathcal{F}^{u}_{f^{n_{j}L}x}(f^{n_{j}L}x,\ell
r)\subseteq\mathcal{F}^{u}_{f^{kL}x}(f^{kL}x,e^{(-\lambda_{u}+4\varepsilon)(n_{j}-k)L}\ell
r)$$ for $0\leq k<n_{j}, j=1,2 \cdots.$ In particular,
$$f^{-n_{j}L}\mathcal{F}^{u}_{f^{n_{j}L}x}(f^{n_{j}L}x,\ell r)
\subseteq\mathcal{F}^{u}_{x}(x,e^{(-\lambda_{u}+4\varepsilon)Ln_{j}}\ell
r),\quad j=1,2 \cdots.$$

For $y\in B_{+\infty}(x,r)$, $f^{i}(y_{u})=(f^{i}y)_{u}$, $\forall i
\in  \mathbb{N}$, thus $y_{u}=f^{-n_{j}L}(f^{n_{j}L}y)_{u}\in
\mathcal{F}^{u}_{x}(x,e^{(-\lambda_{u}+4\varepsilon)Ln_{j}}\ell r)$,
 $\forall j\in \mathbb{N}$. Therefore $y_{u}$ belongs to the intersection of all
 $\mathcal{F}^{u}_{x}(x,e^{(-\lambda_{u}+4\varepsilon)Ln_{j}}\ell r)$ over all $j$ which reduces to $\{x\}$.
Analogously, for $y\in B_{-\infty}(x,r)$, we get that $y_{s}=x$.
Thus $y\in B_{\pm\infty}(x,r)$ implies that  $y=x$.

To conclude, we have obtained the following :
\begin{Cla}\label{claim1}
For any $\sigma> 0 $ there exist $r>0$ and $\rho>0 $ satisfying the
following property. For any $\nu\in B_{\rho}(\mu)\cap\mathcal
M_{inv}(\Lambda)$, there exists a  Borel set $T$ with
$\nu(T)>1-\sigma$ such that
$$ B_{\pm\infty}(x,r)=\{x\},\quad\forall x\in T   .$$
\end{Cla}

Claim \ref{claim1} says that, fixing a small $r>0$, for  $\nu$ close
to $\mu$, $h^{*}_{loc}(x,r,f)=0$  on  a set with large
$\nu$-measure. To estimate the difference between $h_{\mu}(f)$ and
$h_{\mu}(f,\xi)$,  by Lemma \ref{lem3} we need to deal with
$h_{loc}(x,r,f)$. One always has  that $h^{*}_{loc}(x,r,f)\leq
h_{loc}(x,r,f)$ but the inverse inequality is generally not true.
However, we are going to show that,  $h_{loc}(x,r,f)$ is still small
on a set with large measure. Combining Claim \ref{claim1}
 with Lemma \ref{lem3} we aim to  deduce the following proposition.

\begin{Prop}\label{claim2}
Let $\mu\in \cM_{inv}(\Lambda) $ and  $\tau >0.$  There exist $ r>0$
and $ \rho >0 $ such that
$$h_{\nu}(f)-h_{\nu}(f,\xi)\leq\tau$$  for any
$\nu\in B_{\rho}(\mu)\cap \cM_{inv}(\Lambda)$ and any finite
partition $\xi$ with $diam(\xi)\leq r.$
\end{Prop}

\begin{proof}
Let $C_{0}=\sup\limits_{x\in
\Lambda}\{\|D_{x}f\|+1\}$ and $C=h_{top}(f,\Lambda)$. It is clear
that both of them are finite. We assume that $C>0$, otherwise the
entropy map for $f$ is upper semi-continuous and we complete the
proof. Take $\eta=\dfrac{\tau}{2C}$, $\gamma=\dfrac{\tau}{2\log
C_0}$ in Lemma \ref{lem1}, we get $\sigma=
\dfrac{\eta\gamma}{2}\leq\dfrac{\eta}{2}$. By Claim \ref{claim1} we
get $r>0$ and  $\rho>0$ with the property that for any $\nu\in
B_{\rho}(\mu)$ there exists a Borel set $T=T(\nu)$ with
$\nu(T)>1-\sigma$ such that
$$h_{loc}^{*}(x,r,f)=0,\,\,\,\,x\in T.\eqno (3.1)$$
 We could assume that $T$ is compact by the regularity
of measure.

 For any $\delta>0$, $\beta>0$ and $x\in T$, by (3.1) we
get $m(x)>0$, $N(x)>0$ as well as an open neighborhood $V(x)$ of $x$
such that $\forall y\in V(x),$ $$ r_{m(x)}(B_{\pm
N(x)}(y,r),\frac{\delta}{4})\leq e^{\beta m(x)}.$$  By compactness
of $T$, $\exists \,\,x_{1},\cdots ,x_{s}$ such that $T\subseteq
\underset{i=1,\cdots,s}\cup V(x_{i}):= W$.  Then $\nu(W)>1-\sigma.$
By Lemma\ \ref{lem1}, $\nu(\widetilde{W})>1-\frac{\eta}{2}$, where
$\widetilde{W}=\{x:\bar f _{W}(x)>1-\gamma\}$. Denote
$m_{i}=m(x_{i}),$  $N_{i}=N(x_{i})$,  $N_{0}=\max\limits_{1\leq
i\leq s}\{m_{i},N_{i}\}$  and
$W^{\prime}=f^{-N_0}(W)\cap\widetilde{W}$. Since
$\sigma<\dfrac{\eta}{2}$, we get that $\nu(W^{\prime})>1-\eta$. For
 $x\in W^{\prime}$ take $n>2N_0$ large enough such that $\sharp\{0\leq
i<n:f^{i}(x)\in W\}>(1-\gamma)n.$ We define a sequence
$0=n_{0}<n_{1}<\cdots <n_{k-1}<n_{k}=n$ of integers as follows. Let
$n_{1}=N_{0}$ then  $f^{n_{1}}(x)\in W$. Assume that $N_0\leq
n_{i}<n-N_0$ has been defined with $f^{n_{i}}(x)\in W.$   There
exists $x_{i_{j}}$ such that $f^{n_{i}}(x)\in V(x_{i_{j}})$, and
then  we take
$$n_{i+1}=\min\Big\{\min \{k:k\geq n_{i}+m_{i_{j}}, f^{k}(x)\in W\},n\Big\}.$$
If  $\exists\ i$ such that $n-N_{0}\leq n_{i+1}<n $(Case(b)), we take $k=i+2$ and $n_k=n.$
Then the sequence $\{ n_i \}_{i=0}^{k}$ is
$$\{0=n_0< n_1=N_0< \cdots<n_{k-2}< n-N_0\leq n_{k-1} <n_k=n   \}.$$
Otherwise(Case(a)),  the sequence is
$$\{0=n_0< n_1=N_0< \cdots<n_{k-1}< n-N_0 <n_k=n   \},$$
see the figure below.

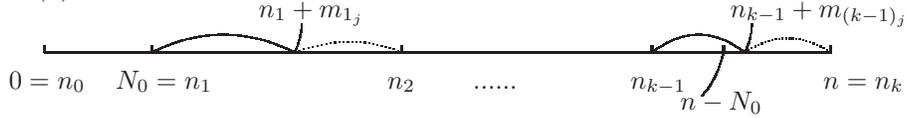
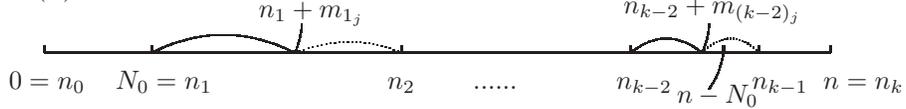
\setlength{\unitlength}{0.95mm}
 \begin{center}
 \begin{figure}[h]
 \begin{picture}(180,40)

 \thicklines
 \put(10,30){\line(1,0){110}}
 \put(10,30){\line(0,1){1}}
 \put(5,25){$0=n_0$}
 \put(25,30){\line(0,1){1}}
 \put(20,25){$N_0=n_1$}
 \put(40,35){$n_1+m_{1_j}$}
 \put(60,30){\line(0,1){1}}
 \put(58,25){$n_2$}
 \put(95,30){\line(0,1){1}}
 \put(92,25){$n_{k-1}$}
 \put(105,30){\line(0,1){1}}
  \put(99,22){$n-N_0$}
 \put(106,35){$n_{k-1}+m_{(k-1)_{j}}$}

 \put(120,30){\line(0,1){1}}
 \put(119,25){$n=n_k$}
 \put(70,25){$......$}
\thinlines
\qbezier(45,30)(45.9,32)(46,34)
\qbezier(108,30)(108.9,32)(109,34)
 \qbezier(105,30)(104,26)(102,24)

 \qbezier(25,30)(35,35)(45,30)
 \qbezier[25](45,30)(53,33)(60,30)
 \qbezier(95,30)(101.5,35)(108,30)
 \qbezier[25](108,30)(114,34)(120,30)

\thicklines
 \put(10,7){\line(1,0){110}}
 \put(10,7 ){\line(0,1){1}}
 \put(5, 2){$0=n_0$}
 \put(25,7){\line(0,1){1}}
 \put(20, 2){$N_0=n_1$}
 \put(40, 12 ){$n_1+m_{1_j}$}
 \put(60,7 ){\line(0,1){1}}
 \put(58, 2){$n_2$}
 \put(92,7 ){\line(0,1){1}}
 \put(90, 2){$n_{k-2}$}

 \put(91,12.5 ){$n_{k-2}+m_{(k-2)_j}$}
 \put(105,7 ){\line(0,1){1}}

 \put(98.5,0){$n-N_0$}
 \put(110,7 ){\line(0,1){1}}
 \put(109,2){$n_{k-1}$}

 \put(120,7 ){\line(0,1){1}}
 \put(119,2 ){$n=n_k$}
 \put(70, 2){$......$}

 \thinlines
 \qbezier(45,7)(45.9,9)(46,11)
 \qbezier(102,7)(102.9,9)(103,11)
 \qbezier(105,7)(104.5,4)(103,2)

 \qbezier(25,7 )(35,12)(45,7 )
 \qbezier[25](45,7 )(53,10 )(60,7 )
 \qbezier(92,7 )(97,11)(102,7 )
 \qbezier[20](102,7 )(106,11)(110,7 )

 \put(0,38){$\bf{Case (a)}$}
 \put(0,15){$\bf{Case (b)}$}

\end{picture}

\end{figure}
\end{center}

\begin{Rem}
In Case$(a)$ in the figure, although the point $n_{k-1}+m_{(k-1)_{j}}$
is greater than  $n-N_0,$ we point out that it  could be smaller
than or equal to $n-N_0.$ Similarly, the point $n_{k-2}+m_{(k-2)_j}$
could be greater than or equal to $n-N_0$ in Case$(b),$ although it
is smaller than $n-N_0$ in the figure.
\end{Rem}

 Note that for given $ b > 0$ and  $\ell\in \mathbb{N},$ any
ball with radius $b$ could be $(\ell,b)$-spanned by $C_0^{\ell-1}$
points, where $C_{0}=\sup\limits_{x\in \Lambda}\{\|D_{x}f\|+1\}$.
When $i=0$, $n_0=0$, $f^{n_0}(B_{n}(x,r))=B_{n}(x,r)$ could be
$(N_{0},\frac{\delta}{2})$-spanned by $\kappa C_0^{N_0-1}$ points,
where $\kappa$ is  the minimal cardinality of sets which
$(1,\frac{\delta}{2})$-span $M$. When $1 \leq i \leq k-2,$ we have
$N_0 \leq n_i < n-N_{0}$ and $f^{n_{i}}(x)\in
V(x_{i_j})$ for some point $x_{i_{j}}.$ Thus
$$f^{n_{i}}(B_{n}(x,r))\subset B_{\pm
N_{i_j}}(f^{n_i}(x),r),$$ from which  $f^{n_{i}}(B_{n}(x,r))$ could
be $(m_{i_{j}},\frac{\delta}{4})$-spanned by $e^{ m_{i_{j}}\beta}$
points. Since any ball with radius $\frac{\delta}{4}$ could be
$(n_{i+1}-(n_{i}+m_{i_{j}}),\frac{\delta}{4})$-spanned by $C_0^{n_{i+1}-n_{i}-m_{i_{j}}-1}$
points, we get that $f^{n_{i}+m_{i_{j}}-1}(B_{n}(x,r))$ could be
$(n_{i+1}-(n_{i}+m_{i_{j}}),\frac{\delta}{4})$-spanned by
$e^{m_{i_{j}}\beta} C_0^{n_{i+1}-n_{i}-m_{i_{j}}-1}$ points. Thus
$f^{n_{i}}(B_{n}(x,r))$ could be
$(n_{i+1}-n_{i},\frac{\delta}{2})$-spanned by $e^{2m_{i_{j}}\beta}
C_0^{n_{i+1}-n_{i}-m_{i_{j}}-1}$ points.

When $i=k-1$ and for Case$(a)$, just as the
discuss above for $1 \leq i \leq k-2,$ $f^{n_{k-1}}(B_{n}(x,r))$
could be $(n-n_{k-1},\frac{\delta}{2})$-spanned by
$e^{2m_{(k-1)_{j}}\beta} C_0^{n-n_{k-1}-m_{(k-1)_{j}}-1}$ points.
For Case$(b)$, since $(n-n_{k-1})\leq N_0,$
$f^{n_{k-1}}(B_{n}(x,r))$
 could be $(n-n_{k-1}, \frac{\delta}{2})$-spanned by $\kappa C_0^{N_0-1}$ points.

By Lemma 2.1 of \cite{Bowen}, which says that
$r_n(B_n(x,r),\delta, f)\leq\prod\limits_{i=0}^{k-1}r_{n_{i+1}-n_i}(f^{n_i}B_n(x,r),\dfrac{\delta}{2},f),$
we get that
$$r_{n}(B_{n}(x,r),\delta)\leq
\begin{cases}
\kappa C_0^{N_0+\gamma n}e^{2\beta n},       & \text{when  Case(a),}\\[2mm]
\kappa^2 C_0^{2N_0+\gamma n}e^{2\beta n},    & \text{when Case(b)}.
\end{cases}$$
Therefore, for both cases, $\forall x\in W^{\prime}$, $
\forall\delta>0$, $r_{n}(B_{n}(x,r),\delta)\leq \kappa^2
C_0^{2N_0+\gamma n}e^{2\beta n}$ for any $n$ large enough. Thus,
\begin{eqnarray*}
& &\limsup\limits_{n\rightarrow+\infty}\frac{1}{n}\log
r_{n}(B_{n}(x,r),\delta)\\[2mm]
&\leq&
\lim\limits_{n\rightarrow+\infty}(\frac{1}{n}\log\kappa^2+2\beta+\frac{2N_0}{n}
\log C_0+\gamma\log C_0)\\[2mm]
&=&2\beta+\gamma\log C_0.
\end{eqnarray*}
By the choice of $\gamma$ and the arbitrariness of $\beta$, we get
that
$$\limsup\limits_{n\rightarrow+\infty}\frac{1}{n}\log
r_{n}(B_{n}(x,r),\delta)\leq \frac{\tau}{2},\,\,\forall x\in
W^{\prime}, \forall \delta>0.$$ Therefore,
$$h_{loc}(x,r,f)=\lim\limits_{\delta\rightarrow0}\limsup\limits_{n\rightarrow+\infty}\frac{1}{n}
\log r_{n}(B_{n}(x,r),\delta)\leq \frac{\tau}{2}, \quad\quad\forall
x\in W^{\prime}.$$ For any measurable partition $\xi$ with
$diam(\xi) \leq r$ by Lemma \ref{lem3} it holds that
$$h_{\nu}(f)-h_{\nu}(f,\xi)\leq \int h_{loc}(x,r,f) d\nu(x)$$
and thus
\begin{eqnarray*}
h_{\nu}(f)-h_{\nu}(f,\xi)&\leq& \int _{W^{\prime}}h_{loc}(x,r,f)
d\nu(x)+\int_{ M\setminus
W^{\prime}}h_{loc}(x,r,f)d\nu(x)\\[2mm]
                         &\leq& \frac{\tau}{2}+\eta\cdot C \leq\frac{\tau}{2}+\frac{\tau}{2}\leq
                         \tau.
 \end{eqnarray*}
This completes the proof of Proposition \ref{claim2}.\end{proof}

We are now turning to the proof of Theorem \ref{thm1}.\\

\noindent {\it Proof of Theorem \ref{thm1}.}~~~\,For any given
$\mu\in \cM_{inv}(\Lambda)$ we will show that the entropy map is
upper semi-continuous at $\mu.$    For $\mu$ and a real $\tau>0$, we
can choose  two constants   $ r>0, \,\,\rho>0$ as in Proposition
\ref{claim2} and  a partition $\xi$  with  $diam(\xi)<r$ and
$\mu(\partial\xi)=0$. From Proposition \ref{claim2}, we know that
$$h_{\nu}(f)-h_{\nu}(f,\xi)\leq \tau,\,\,\,\,\forall \nu\in B_{\rho}(\mu)\cap \cM_{inv}(\Lambda).$$
Since $h_{\mu}(f)=\sup\limits_{\xi}h_{\mu}(f,\xi)$, we can shrink
$diam(\xi)$  if necessary such that
$$h_{\mu}(f,\xi)-h_{\mu}(f)\leq\tau.$$ Note for a fixed $n$ and a
partition $\xi$ with $\mu(\partial\xi)=0$,
$\dfrac{1}{n}H_{\nu}(f,\xi^n)$ is continuous at $\mu$. Thus
$h_{\nu}(f,\xi)=\inf\limits_{n}\dfrac{1}{n}H_{\nu}(f,\xi^n)$  is
upper semi-continuous at $\mu.$  Shrink $\rho>0$ if necessary, we
get
$$  h_{\nu}(f,\xi)-h_{\mu}(f,\xi)\leq \tau,\,\,\,\,\nu\in \cM_{inv}(\Lambda)
\cap B_{\rho}(\mu).$$ Therefore,
\begin{eqnarray*}
h_{\nu}(f)-h_{\mu}(f)&=&    (h_{\nu}(f)-h_{\nu}(f,\xi))+(h_{\nu}(f,\xi)-h_{\mu}(f,\xi))+(h_{\mu}(f,\xi)-h_{\mu}(f))\\[2mm]
                     &\leq&\tau+\tau+\tau\\[2mm]
                     &\leq&3\tau, \,\,\,\,\nu\in \cM_{inv}(\Lambda)
                     \cap B_{\rho}(\mu),
\end{eqnarray*}
which shows that the entropy map is upper semi-continuous at $\mu.$
\qed

\section{$C^{r}$ $(2\leq r<\infty)$ diffeomorphisms without domination}\label{diffeomorphisms without domination}

In this section, by some brief analysis of the techniques in
\cite{New2} we illustrate the examples of $C^{r}$ $(2\leq r<\infty)$
nonuniformly hyperbolic system without domination for which the
entropy map is not upper semi-continuous. For a  detail proof, readers may refer to \cite{New2}.

We denote $C^{r}(M)(2\leq r<+\infty)$ as the set of $C^{r}$
diffeomorphisms on a smooth surface $M$. We can choose an open
subset $\cU \subset C^{r}(M)$ such that each $f$ in it has a
hyperbolic basic set $\Delta(f)$ with the same adapted neighborhood
$U\subset M$ which has persistent homoclinic tangencies, i.e. there
exist $x,y\in\Delta(f)$ such that $W^{s}(x)$ and $W^{u}(y)$ have
tangencies. This is  according to  Chapter 6 of \cite{PT}.
 Let
$\widetilde{H}_{n}(f)$ be the set of periodic hyperbolic points $p$
which are homoclinic related to $\Delta(f)$ (i.e.
$W^{s}(\Delta(f))\backslash \Delta(f)$ and
$W^{u}(O(p))\backslash O(p)$ have nonempty transverse
intersections and vice versa) with least period less
than or equal to $n$, and let
$\widetilde{H}(f)=\bigcup\limits_{n=1}^{+\infty}\widetilde{H}_{n}(f)$.
Let $\widetilde{\tau}(f)$ be the least integer $n$ such that
$\widetilde{H}_{n}(f)\neq\emptyset$ and let $\mathcal D_{m} $ be the
subset of $\cU$ such that $\widetilde{\tau}(f)=m$. For $p\in
\widetilde{H}(f)$, denote
$\chi(p)=\frac{1}{\pi(p)}\min\{\log|\lambda_{s}^{-1}(p)|,\log|\lambda_{u}(p)|\}$,
where $|\lambda_{s}(p)|<1$ and $|\lambda_{u}(p)|>1$ are the norms of
the two eigenvalues of $D_{p}f^{\pi(p)}$ respectively, and $\pi(p)$
 the least period of $p$. Let $\mu_{p}$ be the periodic measure
supported on $p$, i.e.
$\mu(p)=\dfrac{1}{\pi(p)}\sum\limits_{i=1}^{\pi(p)-1}\delta_{f^{i}(p)}$.
For an ergodic hyperbolic measure $\mu$ on $M$, let
$\chi(\mu)=\min\{|\chi_{s}(\mu)|, |\chi_{u}(\mu)|\}$, where
$\chi_{s}(\mu),\chi_{u}(\mu)$ are the two Lyapunov exponents of
$\mu$. In the sequel, by
$(C^{r},\epsilon)$-perturbation  we mean that the perturbation  is done in the  $\epsilon$-neighborhood in $C^{r}$ topology. By $C^{r}$
perturbation, we mean $(C^{r},\epsilon)$-perturbation for any
sufficiently small $\epsilon$.

Let $f\in \mathcal D_m,$ $n\geq m,$ for any $p\in
\widetilde{H}_{n}(f)$, we first $C^r$-perturb $f$ to get a
homoclinic tangency for $O(p)$ (see  Lemma 8.3 and Lemma 8.4 in \cite{New0}). By a
$C^r$-perturbation  we assume $p$ is both  $r-$shrinking meaning
$|\lambda_s(p)\lambda_u^r(p)|<1$  and nonresonant meaning  that for
any pair of positive integers $n$ and $m$ the number
$|\lambda_{s}^{n}(p)\lambda_{u}^{m}(p)|$ is different from 1.
 Then according to Proposition 5 and Lemma 3 in
\cite{Kaloshin} by a further $C^r$ small perturbation, one can
 get an interval $I$ of tangencies between $W^u(p)$ and $W^s(p)$. Near  this interval we
take one more $C^r$ small perturbation $g$ to create a curve
$J\subset W^u(p,g)$  with $N$ bumps as in Figure 1.

\setlength{\unitlength}{1.5mm}
 \begin{center}
 \begin{figure}[h]\label{small horseshoe figure}
 \begin{picture}(180,45)

\thicklines \put(20,10){\line(1,0){35}}
 \put(25,0){\line(0,20){25}}
 \qbezier[45](25,25)(25,40)(45,45)
 \put(23,8){$p$}

 \thinlines
 \qbezier(23,15)(24,16)(25,17)
 \qbezier(27,15)(26,16)(25,17)
 \qbezier(23,3)(24,2)(25,1)
 \qbezier(27,3)(26,2)(25,1)
 \qbezier(21,12)(22,11)(23,10)
 \qbezier(21,8)(22,9)(23,10)
 \qbezier(29,12)(28,11)(27,10)
 \qbezier(29,8)(28,9)(27,10)

 \thicklines
 \qbezier[50](45,45)(70,50)(60,12)
 \qbezier[30](50,6)(55,0)(60,12)
 \thinlines
 \qbezier(47,10)(47.5,8.5)(49.7,6.4)
 \qbezier(45,10)(46,15)(47,10)
 \qbezier(43,10)(44,5)(45,10)
 \qbezier(41,10)(42,15)(43,10)
 \qbezier(39,10)(40,5)(41,10)
 \qbezier(37,12)(38.4,13.5)(39,10)
 \qbezier[15](36,8)(36.5,12)(37,12)
 \qbezier(34,10)(34.5,8)(34.7,7.5)
 \qbezier(32,10)(33,15)(34,10)

 \qbezier(47,12)(47.5,10)(49,8)
 \qbezier(45,12)(46,15)(47,12)
 \qbezier(43,12)(44,5.7)(45,12)
 \qbezier(41,12)(42,14.5)(43,12)
 \qbezier(39,12)(40,5.7)(41,12)
 \qbezier(36.5,12.3)(38,14)(39,12)
 \qbezier[15](35.6,9)(36,12)(36.5,12.3)
 \qbezier(34,12)(34.5,9)(35,8)
 \qbezier(31.5,10)(33,15)(34,12)

 \qbezier(47,13.5)(47,13.5)(49,9)
 \qbezier(44.5,12)(45.5,16)(47,13.5)
 \qbezier(43.6,11)(43.9,7.5)(44.5,12)
 \qbezier(40.5,11.5)(42,18)(43.6,11)
 \qbezier(39.5,12)(39.8,8)(40.5,11.5)
 \qbezier(36.3,13)(38,16)(39.5,12)
 \qbezier[20](35.3,10)(35.8,12)(36.3,13)
 \qbezier(33.5,14)(34.5,12)(35,9)
 \qbezier(31.5,11)(32.5,14.5)(33.5,14)

 \put(49,8){\line(0,1){1}}
 \put(31.5,10){\line(0,1){1}}

 \put(30,10.7){\line(1,0){20}}
 \put(30,11.5){\line(1,0){20}}
 \put(30,10.7){\line(0,1){0.8}}
 \put(50,10.7){\line(0,1){0.8}}

 \put(30,9.7){\line(0,1){0.5}}
 \put(50,9.7){\line(0,1){0.5}}
 \put(29.8,8.5){$a_1$}
  \put(50,8.5){$a_2$}

 \qbezier(49,11.5)(50,12)(51,13)
 \put(51 ,14){$D_N$}
 \qbezier(37,13.5)(36.8,14)(36.6,16)
 \put(35,17){$g^{k+T}(D_N)$}

 \qbezier(37,10)(37.2,9)(37.4,8)
 \put(37,5){$I$}
 \qbezier(40,7.5)(42,5)(43,4.5)
 \put(42.5,3.5){$J$}

 \put(24.5,17.5){\line(1,0){1}}
 \put(24.5,23){\line(1,0){1}}
 \qbezier(25,19)(24,19.5)(23,21)
 \put(17,23){$g^{-T}(J)$}

 \put(26,17.5){\line(1,0){1}}
  \put(26,23){\line(1,0){1}}
  \put(26,17.5){\line(0,1){5.5}}
  \put(27,17.5){\line(0,1){5.5}}
  \qbezier(27,20)(28,20.5)(28.5,22)
  \put(28,23){$g^{k}(D_N)$}

  \thicklines
  \qbezier[20](61,28)(61,20)(59,15)
  \qbezier[20](62,29)(62,20)(60,16)
  \qbezier[3](61,28)(61.5,27.8)(62,29)
  \qbezier[3](59,15)(59.5,15.5)(60,16)

\thinlines
 \qbezier(36,42)(35,42.5)(34,45)
 \put(32,45){$W^{u}(p,g)$}
 \qbezier(26,10)(28,5)(30,3)
 \put(30,2){$W^{s}(p,g)$}

\end{picture}
 \caption{Creation of small  basic sets}
 \end{figure}
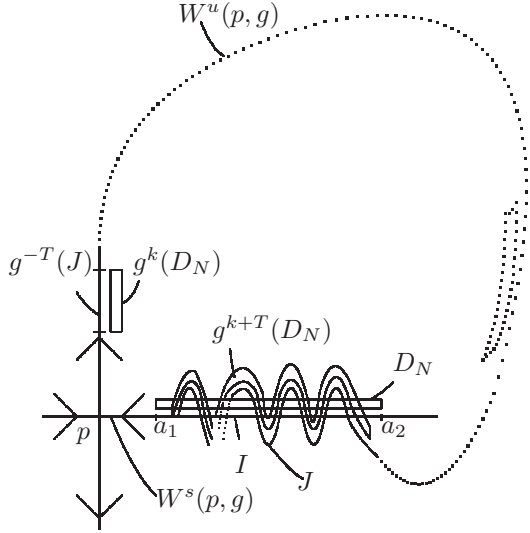
\end{center}

This perturbation can be done as follows. Denote $I=\{a_{1}\leq x
\leq a_{2},y=0 \}$, $J=\{a_{1}\leq x \leq
a_{2}:y=A\cos\omega(x-c)\}.$ To keep the perturbation to be
$C^r$-small, we only require that $A\cdot\omega^r\leq\epsilon,$
where $\epsilon$ could be arbitrarily small. For any fixed small
$\epsilon>0,$
 let $A=\epsilon(\dfrac{a_{2}-a_{1}}{\pi N})^{r}$,
 $\omega=\dfrac{\pi\ N}{a_{2}-a_{1}}$, $c=\dfrac{a_{1}+a_{2}}{2}.$
 Without loss of generality, we assume that $p$ is a fixed point.

 To create small hyperbolic basic sets,  consider a small rectangle $D_N$ close to $I$ with distance less than $\dfrac{A}{4}$ and consider the iterations of $g^{k+T}$
 (where $k$ denote the first $k$ iterations near $p$). To obtain a periodic hyperbolic basic set $\Delta(p,N)$
  by transversal intersections, it is required that
  $$A\cdot |\lambda_u|^k\gtrsim1,\quad|\lambda_s|^k\lesssim A,\eqno(4.1)$$
   where by $a\gtrsim b$
 we mean that $a\geq const\cdot b$,
 and  the $const$ is independent of $N$ and $k(N)$ ( $a\lesssim b $ is
 defined similarly), and by $a\simeq b$ we mean that $a\gtrsim b$ and $a\lesssim b.$
 In this way we get an $N$-horseshoe with topological entropy $\dfrac{\log N}{k+T}.$
Note that $A=\epsilon(\dfrac{a_{2}-a_{1}}{\pi N})^{r}\simeq \dfrac{1}{N^r},$
 so to get (4.1), $k$ should be large enough such that
 $$k\simeq\dfrac{-\log A}{\chi(p)}=
 \dfrac{\log (N^r\cdot const)}{\chi(p)}=\dfrac{r\log N}{\chi(p)}+const $$ and
 thus $$h_{top}(\Delta(p,N),g)
 =\dfrac{\log N}{\dfrac{r\log N}{\chi(p)}+const+T}\,\,.$$

 For $n\in \mathbb{N},$ choose $N(n)$ large enough such that   $$h_{top}(\Delta(p,N(n)),g)>\dfrac{\chi(p)}{r}-\dfrac{1}{n}.\,\,$$
 By Variational Principle,
  there exists an ergodic measure $\nu_n$ supported on $\Delta(p,N(n))$
  such that $h_{\nu_{n}}(g)> \dfrac{\chi(p)}{r}-\dfrac{1}{n}.$
By estimating one sees  that $Dg^{k+T}$   expands unstable direction
in $\Delta(p,N)$ about $N$ times and contracts  the stable direction
about $1/N$ times,  so $\chi(\mu_n)$ of any ergodic measure $\mu_n$
on $\Delta(p,N(n))$ will be close to $\dfrac{\log N}{k+T}
\simeq\dfrac{\chi(p)}{r}$. Moreover, since by iterations of $g$,
$\Delta(p,N)$ spends most of time near $p,$ $\mu_n$ is close to the
periodic measure $\mu_p$.  Let $N(n)$ be larger if necessary such
that $$\chi(\mu_n)> \dfrac{\chi(p)}{r}-\dfrac{1}{n}\quad
\text{and}\quad d(\mu_{n},\mu_{p})<\dfrac{1}{n}.$$ Denote
$\widetilde{\Delta}(p,n)=\Delta(p,N(n)).$

To conclude, for any diffeomorphism $f\in \mathcal D_m,$ any positive integer $n\geq m,$ and any $p\in \widetilde{H_{n}}(f),$ through any $C^r$ small perturbation we get a diffeomorphism $g_n$ satisfying property $\mathcal{S}_n$:

(1) \ there exists a hyperbolic basic  set  $\widetilde{\Delta}(p,n)$ and an  ergodic measure $\nu_{n}$ on $\widetilde{\Delta}(p,n)$
such that
$$h_{\nu_{n}}(g_n)> \dfrac{\chi(p)}{r}-\dfrac{1}{n},$$

(2) \ for any ergodic measure $\mu_{n}$ on $\widetilde{\Delta}(p,n)$, we have
$$\chi(\mu_{n})>\dfrac{\chi(p)}{r}-\dfrac{1}{n} \ \ \text{and}\ \ d(\mu_{n},\mu_{p})<\dfrac{1}{n}
    .$$

Denote $\mathcal D_{m,n}(n\geq m)$ as the subset of $\mathcal D_{m}$
satisfying property $\mathcal S_{n}.$ It's obvious that property $\mathcal{S}_n$ is an open property. From above discussion, we see that $\mathcal D_{m,n}$ is an open dense subset of $\mathcal D_m.$  Let
$$\mathcal
R=\bigcup\limits_{m=1}^{+\infty}\bigcap\limits_{n=m}^{+\infty}\mathcal
D_{m,n},$$ then $\mathcal R$ is a residual subset of $\mathcal U.$
For any $f\in\mathcal R$, any $p\in \widetilde H_{n}(f)$, there
exists a sequence of ergodic measures $\{\nu_{n}\}$ such that
$\nu_{n}\rightarrow \mu_{p}$  and
$\chi(\nu_{n})>\dfrac{1}{2r}\chi(p)$. By Definition \ref{def1},
$\{\nu_{n}\}$ and $\mu_{p}$ are supported on a Pesin set
$\Lambda(\dfrac{1}{2r}\chi(p),\dfrac{1}{2r}\chi(p);\varepsilon)$.
But at the same time, $h_{\nu_{n}}(f)\rightarrow
\dfrac{1}{r}\chi(p)$ while $h_{\mu_{p}}(f)=0$, which implies
that the entropy map of $f\in \mathcal R$ is not upper semi-continuous
at $\mu_p$ on the Pesin set
$\Lambda(\dfrac{1}{2r}\chi(p),\dfrac{1}{2r}\chi(p);\varepsilon)$.

Note that although for each fixed $n,$ $\nu_n$ is supported on $\widetilde{\Delta}(p,n)$
 which is uniformly hyperbolic and the angles between $E^s$ and $E^u$ is
  uniformly bounded below by about $A\cdot \omega\simeq\dfrac{1}{(N(n))^{r-1}},$
the angles of the Oseledec splittings for the sequence $\nu_n$,
$n\geq 1$, may be arbitrary small as $n$ goes to infinity.
 Therefore there is no domination between $E^s$ and $E^u$ over  $\Lambda(\dfrac{1}{2r}\chi(p),\dfrac{1}{2r}\chi(p);\varepsilon)$.


\end{document}